\documentclass[11pt]{article}

\usepackage{amsmath}
\usepackage{amssymb}

\begin{document}
\newtheorem{claim}{Claim}[section]
\newtheorem{theorem}{Theorem}[section]
\newtheorem{corollary}[theorem]{Corollary}
\newtheorem{definition}[theorem]{Definition}
\newtheorem{conjecture}[theorem]{Conjecture}
\newtheorem{question}[theorem]{Question}
\newtheorem{lemma}[theorem]{Lemma}
\newtheorem{proposition}[theorem]{Proposition}
\newtheorem{problem}[theorem]{Problem}
\newenvironment{proof}{\noindent {\bf
Proof.}}{\rule{3mm}{3mm}\par\medskip}
\newcommand{\remark}{\medskip\par\noindent {\bf Remark.~~}}
\newcommand{\pp}{{\it p.}}
\newcommand{\de}{\em}

\newcommand{\JEC}{{\it Europ. J. Combinatorics},  }
\newcommand{\JCTB}{{\it J. Combin. Theory Ser. B.}, }
\newcommand{\JCT}{{\it J. Combin. Theory}, }
\newcommand{\JGT}{{\it J. Graph Theory}, }
\newcommand{\ComHung}{{\it Combinatorica}, }
\newcommand{\DM}{{\it Discrete Math.}, }
\newcommand{\ARS}{{\it Ars Combin.}, }
\newcommand{\SIAMDM}{{\it SIAM J. Discrete Math.}, }
\newcommand{\SIAMADM}{{\it SIAM J. Algebraic Discrete Methods}, }
\newcommand{\SIAMC}{{\it SIAM J. Comput.}, }
\newcommand{\ConAMS}{{\it Contemp. Math. AMS}, }
\newcommand{\TransAMS}{{\it Trans. Amer. Math. Soc.}, }
\newcommand{\AnDM}{{\it Ann. Discrete Math.}, }
\newcommand{\NBS}{{\it J. Res. Nat. Bur. Standards} {\rm B}, }
\newcommand{\ConNum}{{\it Congr. Numer.}, }
\newcommand{\CJM}{{\it Canad. J. Math.}, }
\newcommand{\JLMS}{{\it J. London Math. Soc.}, }
\newcommand{\PLMS}{{\it Proc. London Math. Soc.}, }
\newcommand{\PAMS}{{\it Proc. Amer. Math. Soc.}, }
\newcommand{\JCMCC}{{\it J. Combin. Math. Combin. Comput.}, }

%\begin{document}

\title{Shortest circuit covers of signed graphs}
\date{}

\author{
{Jian Cheng, You Lu, Rong Luo, and Cun-Quan Zhang\thanks{This research project has been partially supported
 by an  NSA grant   H98230-14-1-0154, an NSF grant DMS-1264800}} \\
Department of Mathematics\\
 West Virginia University\\
 Morgantown, WV 26505\\
Email:  \{jiancheng, yolu1, rluo, cqzhang\}@math.wvu.edu \\
}
\date{}
\maketitle

\begin{abstract}
A shortest circuit cover ${\cal F}$ of a bridgeless graph $G$ is a family of circuits that covers every edge of $G$ and is of minimum total length.
The total length of a shortest circuit cover
${\cal F}$  of $G$ is denoted by $SCC(G)$.
For ordinary graphs (graphs without sign), the subject of shortest circuit cover is closely related to some mainstream areas, such as, Tutte's integer flow theory, circuit double cover conjecture, Fulkerson conjecture, and others.
For signed graphs $G$, it is proved recently by
 M\'a\v{c}ajov\'a,   Raspaud,    Rollov\'a and   \v{S}koviera that
$SCC(G) \leq 11|E|$ if $G$ is s-bridgeless, and
$SCC(G) \leq 9|E|$ if $G$ is $2$-edge-connected.
 In this paper this result is improved as follows,
$$SCC(G) ~ \leq ~ |E| + 3|V| +z$$ where
$z ~=~ \min \{ \frac{2}{3}|E|+\frac{4}{3}\epsilon_N-7,~
|V| + 2\epsilon_N -8\}$ and
$\epsilon_N$ is the negativeness of $G$.
The above upper bound can be further reduced if $G$ is $2$-edge-connected  with  even negativeness.

\end{abstract}

\maketitle

%%%%%%%%%%%%%%%%%%%%%%%%%%%%%%%%%%%%%%%%%%%%%%%%%%%%%
\section{Introduction}
For terminology and notations not defined here we follow \cite{Bondy2008, Diestel2010, West1996}.
%% , Zhang1997,Zhang2012}.
Graphs considered in this paper may have multiple edges or loops. A {\em circuit cover} of a bridgeless
 graph $G$ is a family $\cal C$ of circuits such that each edge of $G$ belongs to at least one member of $\cal C$. The {\em length} of $\cal C$ is the total length of circuits in $\cal C$. A minimum length of a circuit cover of $G$ is denoted by $SCC(G)$.

For  ordinary graphs (graphs without sign), the subject of shortest circuit cover is
 not only a discrete optimization problem
 \cite{Itai1978}, but also
 closely related to some mainstream areas in graph theory,
  such as, Tutte's integer flow theory
\cite{Alon1985SIAM, Bermond1983JCTB, FanJGT1994, Jackson1990, Jamshy1987JCTB, Macajova2011JGT,Zhang1990JGT},
circuit double cover conjecture
\cite{Jamshy1992JCTB, Kostochka1995JGT}, Fulkerson conjecture
\cite{FanJCTB1994}, snarks  and graph minors
\cite{Alspach1994, Jackson1994}.
It is proved by
  Bermond,   Jackson and   Jaeger \cite{Bermond1983JCTB}
that
{\em every graph admitting a nowhere-zero $4$-flow has
$SCC(G) \leq \frac{4|E|}{3}$}.
By applying Seymour's $6$-flow theorem
\cite{Seymour1981}
 or Jaeger's $8$-flow theorem
\cite{Jaeger1979},
  Alon and   Tarsi
\cite{Alon1985SIAM},
  and Bermond,   Jackson and   Jaeger \cite{Bermond1983JCTB} proved that
  {\em
every bridgeless graph $G$
 has $SCC(G)\leq \frac{25|E|}{15}$.}
One of the most famous open problems in this area was proposed by
Alon and   Tarsi
\cite{Alon1985SIAM},
that
{\em
every bridgeless graph $G$
 has $SCC(G)\leq \frac{21|E|}{15}$.}
It is proved
by Jamshy and Tarsi \cite{Jamshy1992JCTB}
that
{\em the above conjecture
implies the circuit double cover conjecture.}
The relations between $SCC(G)$ and Fulkerson conjecture, Tutte's $3$-flow and $5$-flow conjectures were studied by Fan, Jamshy,
 Raspaud and Tarsi in
\cite{FanJCTB1994,Jamshy1987JCTB,FanJGT1994}.

\medskip
For signed graphs.
the following
upper bounds for shortest circuit covers were recently estimated in
\cite{Macajova2015JGT}.

\begin{theorem}
\label{TH: MRRS}
{\rm (M\'a\v{c}ajov\'a, Raspaud,  Rollov\'a and \v{S}koviera
\cite{Macajova2015JGT})}
Let $G$
be an s-bridgeless signed graph.

 (1)  In general,
$SCC(G)\leq 11 |E|$.

(2) If $G$
is $2$-edge-connected, then  $SCC(G)\leq 9 |E|$.
\end{theorem}

In this paper, Theorem~\ref{TH: MRRS}
is further improved as follows.

\begin{theorem}
\label{TH: main result}
Let $G$ be an s-bridgeless signed graph with negativeness $\epsilon_N > 0$.

(1)  In general,
$$SCC(G) ~ \leq ~ |E| + 3|V| +z_1, $$ where
$z_1  ~=~ \min \{ \frac{2}{3}|E|+\frac{4}{3}\epsilon_N-7,~
|V| + 2\epsilon_N -8\}$.

(2) If $G$
is
  $2$-edge-connected and $\epsilon_N$ is even,
  then
$$SCC(G) ~ \leq ~ |E| + 2|V| +z_2, $$ where
$z_2  ~=~ \min \{ \frac{2}{3}|E|+\frac{1}{3}\epsilon_N-4,~
|V| +  \epsilon_N -5\}$.
\end{theorem}

Theorem~\ref{TH: main result} is an analog of a result
(Theorem~\ref{CCCC})
by Fan
\cite{FanJCTB1998} that solves a long standing open problem by
Itai and Rodeh
\cite{Itai1978}.

Note that, in a connected s-bridgeless signed graph $G$ with $|E_N(G)|=\epsilon_N$, $G-E_N(G)$ is a connected unsigned graph (by Lemma \ref{cut}), and hence $|E|\ge \epsilon_N+|V|-1$. Therefore  Theorem~\ref{TH: main result} implies that if $G$ is an s-bridgeless signed graph with $\epsilon_N > 0$, then
$$SCC(G)\leq \frac{14}{3}|E|-\frac{5}{3}\epsilon_N-4.$$
This is an analog of a result (Theorem~\ref{CC}) by Alon and Tarsi \cite{Alon1985SIAM} and by
Bermond, Jackson and Jaeger \cite{Bermond1983JCTB}.

\section{Notation and terminology for signed graphs}
\label{SS: Notation}
A  {\it signed graph} is a graph $G$ with a mapping $\sigma: E(G)\to \{1,-1\}$. An edge $e\in E(G)$ is {\em positive} if $\sigma(e)=1$ and {\em negative} if $\sigma(e)=-1$.
The mapping $\sigma$, called {\em signature}, is usually  implicit in the notation of a signed graph and will be specified only when needed.
\label{P: implicit}
 For a subgraph $H$ of $G$, we use $E_N(H)$ to denote the set of all negative edges in $H$.
 A circuit $C$ of $G$ is {\em balanced} if $|E_N(C)|\equiv 0 \pmod 2$, and {\em unbalanced} otherwise.
 A {\em signed circuit} of $G$ is a subgraph of one of the following three types:
\begin{itemize}
 \item[(1)] a balanced circuit;

  \item[(2)] a short barbell, the union of two unbalanced circuits that meet at a single vertex;

   \item[(3)] a long barbell, the union of two disjoint unbalanced circuit with a path that meets the circuits only at its ends.
\end{itemize}
A {\em barbell} is either a short barbell or a long barbell. The {\em length} of a signed circuit $C$ is the number of edges in  $C$.

 \begin{definition}\label{SCC}
 Let $\mathcal{F}$ be a family of signed circuits of a signed graph $G$ and $K$ be a set of some nonnegative integers.
 \begin{itemize}
\item  $\cal F$ is called a {\em signed circuit cover} (resp., signed circuit $K$-cover) of $G$ if each edge $e$ of $G$ belongs to $k_e$ members of $\cal F$ such that $k_e\ge 1$ (resp., $k_e\in K$). In particular, a signed circuit $\{2\}$-cover is also called a {\em signed circuit double cover}.

\item The {\em length}, denoted by $\ell(\cal F)$, of $\cal F$ is the total length of signed circuits in $\cal F$.

\item  $\cal F$ is called a {\em shortest circuit cover} of $G$ if it is a signed circuit cover of $G$ with minimum length.
 The length of a shortest circuit cover of $G$ is denoted by $SCC(G)$.

 \end{itemize}
\end{definition}

  Clearly,  the signed circuit cover of signed graphs is a generalization of the classic circuit cover of graphs.
  By the definition of signed circuit cover, a signed graph has a signed circuit cover if and only if every edge of the signed graph is contained in a signed circuit. Such signed graph is called  {\em s-bridgeless}.

In a signed graph, {\it switching} at a vertex $u$ means reversing the signs of all edges incident with $u$.  It is obvious
(see \cite{Raspaud2011JCTB})
that  the switching operation preserves signed circuits and thus the existence and the length of a signed circuit cover of a signed graph are two invariants under the switching operation.

 \begin{definition}\label{mini-neg-edge}
Let $G$ be a signed graph, and $\mathcal{X}$ be the collection of signed graphs obtained from $G$ by a sequence of switching operations.  The
{\em negativeness} of $G$ is
$$\epsilon_N(G)=\min\{|E_N(G')| : \forall  G'\in \mathcal{X}\}.$$
  \end{definition}

 \begin{definition}
Let $b$ be a bridge of a connected signed graph $G$ and $Q_1, Q_2$ be the two components of $G-b$. The bridge $b$ is called a {\em g-bridge} of $G$ if $\epsilon_N(Q_1)\equiv \epsilon_N(Q_2)\equiv 0 \pmod 2$.
 \end{definition}

Note that a signed graph $G$ is g-bridgeless if and only if every component of $G$ contains no g-bridges, and is s-bridgeless if and only if for each component $Q$ of $G$, $Q$ is s-bridgeless and $\epsilon_N(Q)\neq 1$ (the ``only if " part is proved in \cite{Bouchet1983JCTB} and the ``if'' part is easy).

\section{Lemmas and outline of the proofs}

Since the concept of g-bridge is introduced in Section~\ref{SS: Notation}, the part (2) of
Theorem~\ref{TH: main result} can be revised as follows in a slightly stronger version.

\begin{theorem}
\label{TH: main result g}
Let $G$ be an s-bridgeless signed graph with negativeness $\epsilon_N > 0$.

(1)  In general,
$$SCC(G) ~ \leq ~ |E| + 3|V| +z_1$$ where
$z_1  ~=~ \min \{ \frac{2}{3}|E|+\frac{4}{3}\epsilon_N-7,~
|V| + 2\epsilon_N -8\}$.

(2) If $G$
is
  g-bridgeless and $\epsilon_N$ is even,
  then
$$SCC(G) ~ \leq ~ |E| + 2|V| +z_2$$ where
$z_2 ~=~ \min \{ \frac{2}{3}|E|+\frac{1}{3}\epsilon_N-4,~
|V| +  \epsilon_N -5\}$.
\end{theorem}

The following is the major lemma for
 the proof of
Theorem~\ref{TH: main result g}.

\begin{lemma}
\label{LE: a pair}
Let $G$ be an s-bridgeless
 signed graph with $|E_N(G)|=\epsilon_N(G)$.
Then $G$ has a pair of subgraphs $\{ G_1, G_2 \}$ such that

(1) $E(G_1) \cup E(G_2) = E(G)$,

(2) $G_1$ contains no negative edge and is bridgeless, and

(3) $G_2-E_N(G)$ is acyclic and $G_2$ has a signed circuit
$\{1,2,  \cdots , k\}$-cover, where $k=2$ if $G$ is g-bridgeless  with an even negativeness, and $k=3$ otherwise.
\end{lemma}

Lemma~\ref{LE: a pair} will be proved
 in Section~\ref{SS: main proof}
 after  some preparations in Section~\ref{SS: Barbell}.

The main result, Theorem~\ref{TH: main result g},
 will be proved as a corollary of
 Lemma \ref{LE: a pair}
  in Section~\ref{SS: final proof}.
The following is the outline of the proof.
By (1) of
Lemma~\ref{LE: a pair},
$$SCC(G) ~ \leq ~ SCC(G_1)+ SCC(G_2).$$
Lemma~\ref{LE: a pair}-(3) provides an estimation for $SCC(G_2)$.
For the bridgeless unsigned
 subgraph $G_1$, we use the following
 classical results in graph theory.

\begin{theorem}
\label{CC}
{\rm (Alon and Tarsi \cite{Alon1985SIAM},
Bermond, Jackson and Jaeger \cite{Bermond1983JCTB}}
Let $G$ be a $2$-edge-connected graph. Then $SCC(G)\leq \frac{5}{3} |E|$.
\end{theorem}

\begin{theorem}
\label{CCCC}
{\rm (Fan \cite{FanJCTB1998})}
Let $G$ be a $2$-edge-connected graph. Then $SCC(G)\leq |E|+|V|-1$.
\end{theorem}

%%%%%%%%%%%%%%%%%%%%%%%%%%%%%%%%%%%%%%%%%%%%%%%%%%%%%%%%%%%%%%%%%%%%%%%%%%%%%%%%%%
\section{Signed circuit covers of generalized barbells}
\label{SS: Barbell}
 In this section, we study  signed circuit covers of generalized barbells which play an important role in the proof of Lemma \ref{LE: a pair}.

A graph is {\em eulerian} if it is connected and each vertex is of even degree.
For a vertex subset $U$ of a graph $G$, let $\delta_G(U)$ denote the set of all edges between $U$ and $V(G)-U$.
In a graph, a {\em $k$-vertex} is a vertex of degree $k$.

\begin{definition}\label{g-barbell}
A signed graph $H$ is called a generalized barbell if it contains a set of vertex-disjoint eulerian subgraphs ${\cal B}=\{B_1, \cdots, B_t\}$ such that
\begin{itemize}
\item[(1)] The contracted graph $X= H/ (\cup_{i=1}^tB_i)$ is acyclic and
\item[(2)] For each vertex $x$ of $X$ (if $x$ is a contracted vertex, then let  $B_x$ be the corresponding eulerian subgraph of ${\cal B}$;
otherwise, simply consider $E(B_x)$ as an empty set),
\begin{equation*}\label{EQ: even}
|E_N(B_x)| \equiv |\delta_H(V(B_x))|  \pmod{2}.
\end{equation*}
\end{itemize}
\end{definition}

 We first study signed eulerian graphs with  even number
  of negative edges which is a  special case of generalized barbells.

 Let $T$ be a closed eulerian trail of a signed eulerian graph. For any two vertices $u$ and $v$ of $T$, we use  $uTv$  to denote the subsequence of $T$  starting with $u$ and ending with $v$ in the cyclic ordering induced by $T$.

\begin{lemma}\label{SCDC of eulerian graph}
Every signed eulerian graph with even number of negative edges has a signed circuit double cover.
\end{lemma}

\begin{proof}
Let $B$ be a counterexample to Lemma~\ref{SCDC of eulerian graph} with $|E(B)|$ minimum.
Then the maximum degree of $B$ is at least $4$  otherwise $B$ is a balanced circuit. By the minimality of $B$,   $B$  cannot be decomposed into two signed eulerian subgraphs, each  contains an even number of negative edges. Thus we have the following observation

{\bf Observation.}  For any eulerian trail $T=u_1e_1u_2 e_2\cdots u_m e_m u_1$ of $B$ where $m = |E(G)|$ and for any two integers $i, j\in [1,m]$ with  $i < j$ and  $u_i=u_j$,
{\em $u_iTu_j$ is a signed eulerian graph with  odd number of negative edges. }

Pick an arbitrary eulerian trail  $T=u_1e_1u_2 e_2\cdots u_m e_m u_1$.
We consider the following two cases.

{\bf Case 1.} For any two integers $i\neq j\in [1,m]$,  if $u_i=u_j$, then $|j-i|\equiv1 \pmod m$.

In this case, the resulting graph obtained from $B$ by deleting all loops is either a single vertex or a circuit.
Since $B$ has an even number of negative edges, one can check that $B$ has a signed circuit double cover, a contradiction.

{\bf Case 2.}  There are  two integers $i, j \in [1,m]$ such that   $ 2 \leq j -i \leq m-2$ and $u_i=u_j$.

 Let $B_1=u_iTu_j$ and $B_2=u_jTu_i$.
 Then, by Observation, both $B_1$ and $B_2$ are signed eulerian subgraphs of $B$ with $B=B_1\cup B_2$
 such that $|E(B_k)| \geq 2$ and $|E_N(B_k)|\equiv 1 \pmod 2$ for each $k = 1,2$.

 If $V(B_1)\cap V(B_2)=\{u_i\}$, then for each $k=1,2$,  let $B_k'$  be the resulting graph obtained from $B_k$ by adding a negative loop $e_k'$  at $u_i$.
 Clearly, $B_k'$ remains eulerian, $|E(B_k')| < |E(B)|$, and $|E_N(B_k')|$ is even.  By the minimality of $B$, $B_k'$ has a signed circuit double cover $\mathcal{F}_k$.
 Since $e_k'$ is a negative loop of $B_k'$, it is covered by two barbells, say $C_k^1$ and $C_k^2$, in $\mathcal{F}_k$.
 Let $C^{\ell}=\cup_{k=1}^2(C_k^{\ell}-e_k')$ for each $\ell=1, 2$.
 Since $V(B_1)\cap V(B_2)=\{u_i\}$, both $C^1$ and $C^2$ are two barbells of $B$, and so $B$ has a signed circuit double cover
 $\cup_{k=1}^2(\mathcal{F}_k-\{C_k^1,C_k^2\})\cup \{C^1, C^2\}$, a contradiction.

 If $V(B_1)\cap V(B_2)\neq \{u_i\}$, then there are two integers $s$ and $t$ such that $s\in[i,j]$, $t\notin [i,j]$, and $u_s=u_t$.  By Observation,
 $|E_N(u_sTu_t)| \equiv 1 \pmod 2$.
 Let $T^*$ be a new closed eulerian trail of $B$ obtained from $T$ by reversing the subsequence $u_iTu_j$ in $T$.  Then $E(u_sT^*u_t)$ is the disjoint union of $E(u_iTu_s)$ and $E(u_jTu_t)$ and thus $E_N(u_sT^*u_t)$ is the disjoint union of $E_N(u_iTu_s)$ and $E_N(u_jTu_t)$.
 Since $|E_N(u_iTu_j)| \equiv 1 \pmod 2$ and $|E_N(u_sTu_t))|\equiv 1 \pmod 2$,  $|E_N(u_iTu_s)| \equiv |E_N(u_jTu_t)| \pmod 2$. Therefore
  $|E_N(u_sT^*u_t)|\equiv 0 \pmod 2$,
 a contradiction to Observation.
 This completes the proof of the lemma.
 \end{proof}

The following lemma is a generalization of Lemma \ref{SCDC of eulerian graph}.

\begin{lemma}
\label{SCDC of g-barbell}
Every generalized barbell has a signed circuit double cover.
\end{lemma}

\begin{proof}
Let $H$ be a generalized barbell. Let  $\{B_1, \cdots, B_t\}$ be a set of disjoint eulerian subgraphs of $H$
and $X = H/ (\cup_{i=1}^{t}B_{i})$ as described in Definition~\ref{g-barbell}.
We will prove by induction on $|E(H) - \cup_{i=1}^tE(B_i)|$.

If $E(H)-\cup_{i=1}^tE(B_i)=\emptyset$, then by the definition of generalized barbell,
each component of $H$ is a signed eulerian graph with an even number of negative edges.
Thus $H$ has a signed circuit double cover by Lemma \ref{SCDC of eulerian graph}.

Now assume that $E(H)-\cup_{i=1}^tE(B_i)\neq \emptyset$. Let $uv\in E(H)-\cup_{i=1}^tE(B_i)$
and $H'$ be the new signed graph obtained from $H$ by deleting $uv$ and adding negative loops $e_u$ and $e_v$ at $u$ and $v$, respectively.
By the definition,  $H'$ remains as a generalized barbell. Since $X$ is acyclic,  $H'$ has more components than $H$,
and thus by induction to each component of $H'$, $H'$ has a signed circuit double cover $\cal F'$.
Let  $\{C_u^1,C_u^2\}$ and $\{C_v^1, C_v^2\}$ be the sets of barbells in $\cal F'$ containing $e_u$ and $e_v$, respectively.
Since $e_u$ and $e_v$ belong to two distinct components of $H'$, $C^i=(C_u^i-e_u)\cup (C_v^i-e_v)+uv$ ($i=1, 2$) is a barbell in $H$.
Hence $$(\mathcal{F}'-\{C_u^1, C_u^2, C_v^1, C_v^2\})\cup \{C^1, C^2\}$$
is a signed circuit double cover of $H$.
\end{proof}

\begin{lemma}\label{g-cycle}
Let $H$ be a generalized barbell with a set of vertex-disjoint eulerian subgraphs ${\cal B}=\{B_1, \cdots, B_t\}$,
and assume that $\{B_1,\cdots, B_s\}$ $(2\leq s\leq t)$ is the set of eulerian subgraphs corresponding to the $1$-vertices of  the contracted graph $X=H/(\cup_{i=1}^tB_i)$.
If each $B_i$ $(1\leq i\leq t)$ is a circuit, then there is a family of signed circuits $\mathcal{F}$ in $H$ such that each edge $e$ of $H$ belongs to
\begin{itemize}
\item[(a)] exactly one member of $\mathcal{F}$ if $e\in \cup_{i=1}^sE(B_i)$,

\item[(b)] one or two members of  $\mathcal{F}$ if $e\in \cup_{i=s+1}^tE(B_t)$, and

\item[(c)] at most one member of  $\mathcal{F}$ if $e\in E(H)-\cup_{i=1}^tE(B_i)$.
\end{itemize}
\end{lemma}

 \begin{proof}
Assume that $H$ is embedded
 in the plane and let $\overline{X^*}$ be a graph obtained from  $X$ by first  clockwise splitting each vertex $x$ with  even degree into $\frac{1}{2}d_{X}(x)$ $2$-vertices, and replacing each maximal subdivided edge  with a single edge. Then each vertex of $\overline{X^*}$ is of odd degree.
 By the definition of generalized barbell, $\overline{X^*}$ is a forest and $V(\overline{X^*})$ corresponds to the set of unbalanced circuits of $\mathcal{B}$.
 Thus $\overline{X^*}$ has  a spanning subgraph satisfying that each component is  a star graph with at least two vertices.
 Let $K_{1,r_i}$ ($i = 1,\cdots, \ell$) be all such star subgraphs.

Note that $V(\overline{X^*})=\cup_{i=1}^{\ell}V(K_{1,r_i})$ corresponds to the set of unbalanced circuits of $\mathcal{B}$.
For $1\leq i\leq \ell$, one can check that the subgraph of $H$ corresponding to $K_{1,r_i}$ has a signed circuit cover $\mathcal{F}_i$
such that each edge of the unbalanced circuits corresponding to $1$-vertices of $K_{1,r_i}$ is covered by $\mathcal{F}_i$ exactly once and each edge of the unbalanced circuit corresponding to the unique vertex of $K_{1,r_i}$ with degree $r_i\ge 2$ is covered by $\mathcal{F}_i$ once or twice.
Therefore the union of $\cup_{i=1}^{\ell}\mathcal{F}_i$ together with the set of balanced circuits of $\mathcal{B}$ is a desired family $\mathcal{F}$ of signed circuits of $H$.
 \end{proof}

Given a family of sets $\{A_1,\cdots, A_t\}$, their {\em symmetric difference}, denoted by $\Delta_{i=1}^tA_i$, is defined as the set consisting of
 elements contained in an odd number of $A_i$'s.

 The following result is stronger than Lemma~\ref{SCDC of g-barbell} which states that a generalized barbell has a signed circuit $\{1,2\}$-cover with some edges covered only once.

  \begin{lemma}\label{1,2-cover of g-barbell}
  For each  generalized barbell, it  either
\begin{itemize}
 \item[(i)]  can be decomposed  into balanced circuits, or

\item[(ii)] has  a signed circuit $\{1,2\}$-cover $\mathcal{F}$ such that there are  two edge-disjoint unbalanced circuits $C_1$ and $C_2$  whose edges are covered by  $\mathcal{F}$ exactly once.
  \end{itemize}
  \end{lemma}

\begin{proof}
Let $H$ be a counterexample to Lemma \ref{1,2-cover of g-barbell} with $|E(H)|$ minimum.
Thus $H$ is connected. Otherwise each component of $H$ satisfies either (i) or (ii).
This implies that $H$ satisfies either (i) or (ii), a contradiction to the choice of $H$.

\begin{claim}\label{claim1}
$H$ is eulerian and therefore contains an even number of negative edges.
\end{claim}

\noindent {\em Proof of Claim \ref{claim1}.}  By the definition of generalized barbell, it is sufficient to show that  $H$ is bridgeless.  Suppose to the contrary that $H$ has a bridge.
By Lemma \ref{SCDC of g-barbell}, $H$ has a signed circuit double cover $\mathcal{F}'$.  Since $H$ has bridges, $\mathcal{F}'$ contains a barbell $C$ with two unbalanced circuits $C_1$ and $C_2$.
Then $\mathcal{F}=\mathcal{F}'-\{C\}$ is a signed circuit $\{1,2\}$-cover of $H$ and covers $C_1$ and $C_2$  exactly once, a contradiction.  This proves the claim.~$\Box$

\medskip
Since $H$ is eulerian by Claim \ref{claim1}, $H$ has a decomposition
$$\mathcal{C}=\{C_1,\cdots,C_h, C_{h+1},\cdots,C_{h+m}, C_{h+m+1},\cdots, C_{h+m+n}\},$$
 where $h, m$ and $n$ are three nonnegative integers, and each $C_i$ is an unbalanced circuit if $1\leq i\leq h$, a short barbell if $h+1\leq i\leq h+m$,
 and a balanced circuit otherwise. We choose such a decomposition  that

(a) $h+2m+n$ is as large as possible,

(b) subject to (a), $n$ is as large as possible, and

(c) subject to (a) and (b), $m$ is as large as possible.

\begin{claim}\label{claim2}
$h\ge 2$ is even and $|V(C_i)\cap V(C_j)|=0$ for $1\leq i<j\leq h$.
\end{claim}
\noindent {\em Proof of Claim \ref{claim2}.}  If $h= 0$, then $\cal C$ satisfies (i) if $m = 0$ and  $ \mathcal{C}\setminus\{C_1\}$ satisfies (ii) otherwise.  Thus $h > 0$. Since $|E_N(H)|=\sum_{i=1}^{h+m+n}|E_N(C_i)|$ is even and $|E_N(C_i)|$ is even for $h+1\leq i\leq h+m+n$, we have  that $h\ge 2$ is even.

Let $C_i$ and $C_j$ be two circuits in $\cal C$ with $1\leq i < j\leq h$. If $|V(C_i)\cap V(C_j)| \geq 3$,
then $C_i\cup C_j$ can be decomposed into three or more circuits (balanced or unbalanced), a contradiction to (a).
So $|V(C_i)\cap V(C_j)|\leq 2$. If $|V(C_i)\cap V(C_j)|=2$, then $C_i\cup C_j$ has a decomposition into two balanced circuits
since both $C_i$ and $C_j$ are unbalanced circuits, which contradicts (b).
If $|V(C_i)\cap V(C_j)|=1$, then $C_i\cup C_j$ is a short barbell, which contradicts (c). So the claim is true.~$\Box$

\medskip
Let $H'=H/(\cup_{i=1}^hC_i)$ and for $1\leq i\leq h$, let $c_i$ be the  vertex of $H'$ corresponding to $C_i$.
Let $T'$ be a spanning tree of $H'$ since $H$ is connected. By Claim \ref{claim2}, $h\geq 2$ is even.
Let  $P_j$ ($1\leq j\leq \frac{h}{2}$) be a path in $T'$ from $c_{2j-1}$ to $c_{2j}$ and let
$$F'=T'[\Delta_{j=1}^{\frac{h}{2}}E(P_j)]$$
Then $F'$ is a forest and $\{c_1,\cdots,c_h\}$ is the set of vertices of $F'$ with odd degree. By the definition, the subgraph of $H$ corresponding to $F'$ is a generalized barbell satisfying the conditions in Lemma \ref{g-cycle}, and thus, by Lemma \ref{g-cycle}, it has a family $\mathcal{F}^*$ of signed circuits such that  $\mathcal{F}=\mathcal{F}^*\cup \{C_{h+1}, \cdots, C_{h+m+n}\}$ is a signed circuit $\{1,2\}$-cover of $H$ and at least two unbalanced circuits in $\{C_1,\cdots,C_h\}$ are covered by $\mathcal{F}$ exactly once, a contradiction. This completes the proof of Lemma \ref{1,2-cover of g-barbell}.
\end{proof}

%%%%%%%%%%%%%%%%%%%%%%%%%%%%%%%%%%%%%%%%%%%%%%%%%%%%%%%%%%%%%%%

 \section{Proof of Lemma \ref{LE: a pair}}
\label{SS: main proof}

 In this section, we complete the proof of Lemma \ref{LE: a pair}. For a signed graph $G$, we use $B(G)$ to denote the set of bridges of $G$ and for each $e\in E_N(G)$,  define
 $$S_G(e)=\{e\}\cup \{f : \{e,f\} \mbox{ is a $2$-edge-cut of $G$}\}.$$

 Let $B_g(G)$ be the subset of $B(G)$ such that, for each $b\in B_g(G)$, at least one component of $G-b$ contains an odd number of negative edges, and let $B_s(G)$ be the subset of $B(G)$ such that, for each $b\in B_s(G)$, each component of $G-b$ contains negative edges.
We need the following lemmas.

\begin{lemma}\label{LE: g-bridgeless}
Let $H$ be a signed graph satisfying that $|E_N(H)|\ge 2$ and $H-E_N(H)$ is a spanning tree of $H$. If $|E_N(H)|$ is even, then $H$ has a generalized barbell containing all edges of $B_g(H)\cup (\cup_{e\in E_N(H)}S_H(e))$.
\end{lemma}

\begin{proof}
Let  $T=H-E_N(H)$. Then $E(H)$ is the disjoint union of
$E(T)$ and  $E_N(H)$.
For each $e \in E_N(H)$, let $C_e$ be the unique circuit of $T+e$.

Let $H' = \bigtriangleup_{e \in E_N(H)}C_e$ and
$O_{H'}$ be the set of all components of $H'$ containing an odd number of negative edges. Since $|E_N(H)|$ is even, so is $|O_{H'}|$. Let $\{v_1,v_2,\cdots,v_{2t}\}$ be the set of vertices of the contracted graph $H/H'$ corresponding to $O_{H'}$. For $i=1,\cdots, t$, there is a shortest path $P_i$ in $H/H'$ from $v_{2i-1}$ to $v_{2i}$. Note that $E_N(H)\subseteq E(H')$ and hence $E(P_i)\subseteq E(H/H')\subseteq E(T)$. Since $T$ is a tree of $H$, $H'' =H'\cup (\bigtriangleup_{i=1}^tP_i)$ is a generalized barbell.

For every bridge $b\in B_g(H)$, each component of $H-b$ contains an odd number of negative edges since $|E_N(H)|$ is even, and thus contains an odd number of members of $O_{H'}$. This fact implies that $b$ must belong to an odd number of members of $\{P_1,\cdots,P_t\}$ and thus $b\in E(H'')$. Hence $B_g(H)\subseteq E(H'')$.
For every $e\in E_N(H)$, it is obvious that $S_H(e)\subseteq  E(C_e)$ and $S_H(e)\cap E(C_{f})=\emptyset$ for any $f\in E_N(G)-\{e\}$, which implies that $S_H(e)\subseteq E(H')$. Therefore, $\cup_{e\in E_N(H)}S_H(e)\subseteq E(H')\subseteq E(H'')$.
\end{proof}

\begin{lemma}\label{LE: s-bridgeless}
Let $H$ be a signed graph satisfying that $|E_N(H)|\ge 2$ and $H-E_N(H)$ is a spanning tree of $H$. Then $H$ has a signed circuit $\{0,1,2,3\}$-cover such that each edge of $B_s(H)\cup (\cup_{e\in E_N(H)}S_H(e))$ is covered at least once and each negative loop (if any) is covered precisely twice.
\end{lemma}

\begin{proof}
Let $H$ be a counterexample with $|E(H)|$ minimum.

\begin{claim}\label{CL: bridges}
$B(H)=\emptyset$.
\end{claim}

\noindent {\em Proof of Claim \ref{CL: bridges}.} Suppose to the contrary that $B(H)\neq \emptyset$. Let $b=u_1u_2\in B(H)$ and $Q_1$ and $Q_2$ be the two components of $H-b$ such that $u_i\in Q_i$ for $i=1,2$.

If $b\in B(H)-B_s(H)$, then there is one member in $\{Q_1,Q_2\}$, without loss of generality, say $Q_1$, satisfying that $B_s(Q_1)=B_s(H)$ and $E_N(Q_1)= E_N(H)$. By the minimality of $H$, $Q_1$ (and thus $H$) has a desired signed circuit $\{0,1,2,3\}$-cover, a contradiction.

If $b\in B_s(H)$, then $|E_N(Q_1)|\ge 1$ and $|E_N(Q_2)|\ge 1$.
For each $i=1,2$, let $Q_i^*$ be the graph obtained from $Q_i$ by adding a negative loop $e_i$ at $u_i$. It is easy to see that  $B_s(Q_1^*)\cup B_s(Q_2^*)=B_s(H)-\{b\}$ and $\cup_{i=1}^2(E_N(Q_i^*)-\{e_i\})=E_N(H)$. By the minimality of $H$, each $Q_i^*$ has a signed circuit $\{0,1,2,3\}$-cover $\mathcal{F}_i^*$ which covers each edge of $B_s(Q_i^*)\cup E_N(Q_i^*)$ at least once and covers each negative loop of $Q_i^*$ exactly twice.
Let $C_i^1$ and $C_i^2$ be the two signed circuits in
$\mathcal{F}_i^*$ containing $e_i$. Since $e_i$ is a negative loop, $C_i^j$ ($j=1,2$) is a barbell of $Q_i^*$, and so $C^j=(C_1^j-e_1)\cup (C_2^j-e_2)+b$ is also a barbell of $H$.  Therefore, $\mathcal{F}=(\mathcal{F}_1^*-\{C_1^1,C_1^2\})\cup (\mathcal{F}_2^*-\{C_2^1,C_2^2\})\cup \{C^1,C^2\}$ is a  desired signed circuit $\{0,1,2,3\}$-cover of $H$, a contradiction. $\Box$

\medskip

Claim \ref{CL: bridges} implies that $H$ is $2$-edge-connected. So Lemma \ref{LE: s-bridgeless} follows from Lemmas \ref{LE: g-bridgeless} and \ref{SCDC of g-barbell} if $|E_N(H)|$ is even. Since $|E_N(H)|\ge 2$, in the following, we  assume that $|E_N(H)|\ge 3$ is odd.

Let  $T=H-E_N(H)$. Note that $T$ is a spanning tree of $H$ and $E(H)$ is the disjoint union of
$E(T)$ and  $E_N(H)$.
For each $e \in E_N(H)$, let $C_e$ be the unique circuit of $T+e$.

\begin{claim}\label{CL: s-bridgeless}
For every $e\in E_N(H)$, $H$ has a signed circuit containing all edges of $S_H(e)$.
\end{claim}

\noindent {\em Proof of Claim \ref{CL: s-bridgeless}.} Let $e\in E_N(H)$ and $f\in E_N(H)-\{e\}$. Note that $S_H(e)\subseteq E(C_e)$, $S_H(f)\subseteq E(C_f)$ and $S_H(e)\cap S_H(f)=\emptyset$ (it can be checked easily since $T=H-E_N(H)$ is a spanning tree of $H$). If $|V(C_e)\cap V(C_f)|\leq 1$, then there is a shortest path $P$ in $T$  joining $C_e$ to $C_f$ (note that $P$ is a single vertex if $|V(C_e)\cap V(C_f)|=1$), and so $C_e\cup C_f\cup P$ is a desired signed circuit.  If $|V(C_e)\cap V(C_f)|\ge 2$, since $T$ is a spanning tree of $H$, then $C_{e} \cap C_{f}$ is a path containing no edges of $S_{H}(e)$. Thus $C_e\Delta C_f$ is a balanced circuit as desired.~$\Box$

\begin{claim}\label{CL: 2-edge-cut}
Each edge $e\in E_N(H)$ is contained in a $2$-edge-cut of $H$.
\end{claim}

\noindent {\em Proof of Claim \ref{CL: 2-edge-cut}.}  Suppose to the contrary then there is a negative edge $e\in E_N(H)$ such that $H_0=H-e$ remains $2$-edge-connected. If $H$ contains negative loops, we choose $e$ which is a negative loop.

Since $H_0$ is $2$-edge-connected and $|E_N(H_0)|=|E_N(H)-\{e\}|\ge 2$ is even, Lemma \ref{LE: g-bridgeless} implies that $H_0$ has a generalized barbell $H_1$ containing all edges of $\cup_{f\in E_{N}(H_0)}S_{H_0}(f)$.  Let $\mathcal{F}_1$ be a signed circuit double cover of $H_1$ by Lemma \ref{SCDC of g-barbell}. Note that $S_H(e)=\{e\}$ and $S_H(f)\subseteq S_{H_0}(f)$ for any $f\in E_N(H_0)=E_N(H)-\{e\}$. Thus $\cup_{f\in E_N(G)}S_{H}(f)\subseteq \{e\}\cup (\cup_{f\in E_{N}(H_0)}S_{H_0}(f))$.

If $e$ is not a negative loop of $H$,  then $H$ has no loop, but has a signed circuit $C$ containing $e$  by Claim \ref{CL: s-bridgeless}. Thus $\mathcal{F}=\mathcal{F}_1\cup \{C\}$ is a signed circuit $\{0,1,2,3\}$-cover of $H$ covering all edges of $\cup_{f\in E_N(H)}S_{H}(f)$, a contradiction.

Assume that $e$ is a negative loop of $G$ and let $u$ denote the unique endvertex of $e$.

 If $\mathcal{F}_1$ contains a barbell $C$, then let $C_1$ and $C_2$ be the two unbalanced circuits of $C$. Since $H$ is $2$-edge-connected, there are two edge-disjoint paths in $H$ from $u$ to $C_1$ and $C_2$, denoted by $P_1$ and $P_2$, respectively. Then $C_i'=C_i\cup P_i+e_0$ for $i=1,2$  is a barbell of $H$. Since $\mathcal{F}_1$ is a signed circuit double cover of $H_1$,
$\mathcal{F}=(\mathcal{F}_1-C)\cup \{C_1', C_2'\}$ is a desired signed circuit $\{0,1,2,3\}$-cover of $H$, a contradiction.

 If $\mathcal{F}_1$ contains no barbells, then $e$ is the unique loop of $H$. Note that $H_1$ is a generalized barbell. By Lemma \ref{1,2-cover of g-barbell}, $H_1$ has either a decomposition $\mathcal{F}_1'$ into balanced circuits or a signed circuit $\{1,2\}$-cover $\mathcal{F}_1''$ and two edge-disjoint unbalanced circuit $C_1$ and $C_2$ such that  each edge in $E(C_1)\cup E(C_2)$ is covered by $\mathcal{F}_1''$  exactly once.
In the former case, let $C'$ be a signed circuit containing $e$ by Claim \ref{CL: s-bridgeless}. Then the family $\mathcal{F}=\mathcal{F}_1'\cup \{C',C'\}$ is a desired signed circuit $\{0,1,2,3\}$-cover of $H$.
 In the latter case, since $H$ is $2$-edge-connected, there are two edge-disjoint paths of $H$ from $u$ to $C_1$ and $C_2$, denoted by $P_1$ and $P_2$, respectively.  Similar to the case when $\mathcal{F}_1$ contains a barbell, we can construct a desired signed circuit $\{0,1,2,3\}$-cover of $H$, and thus obtain a contradiction. $\Box$

  \medskip

  By Claim \ref{CL: 2-edge-cut}, $H$ contains no negative loops and $|S_H(e)|\ge 2$ for each $e\in E_N(H)$. For every $e\in E_N(G)$, let $\mathcal{M}_e$ denote the set of all components of the subgraph $H-S_H(e)$.

\begin{claim}\label{claim4}
For two distinct $e, e'\in E_N(H)$, $S_H(e')$ is contained in exactly one member of $\mathcal{M}_e$.
\end{claim}

\noindent {\it Proof of Claim \ref{claim4}.}
Note that each member of $\mathcal{M}_e$ is $2$-edge-connected, and $S_H(e)\cap S_H(e')=\emptyset$ since $H-E_N(H)$ is a spanning tree of $H$.  Then $S_H(e')\subseteq \cup_{M\in \mathcal{M}_e}E(M)$.
Let $e^*$ be an arbitrary edge in $S_H(e')-\{e'\}$.  If there are two distinct members $M_i$ and $M_j$ of $\mathcal{M}_e$ such that $e'\in E(M_i)$ and $e^*\in E(M_j)$,
then both $M_i-e'$ and $M_j-e^*$ are connected, and so $H-\{e',e^*\}$ is also connected.
This contradicts that $\{e',e^*\}$ is a $2$-edge-cut of $H$. So $e'$ and $e^*$ are contained in a common member of $\mathcal{M}_e$.
The arbitrariness of $e^*$ implies that the claim holds.~$\Box$

\medskip
For every $e\in E_N(H)$, let $m_e=\max\{|E_N(H)\cap E(M)| : M\in \mathcal{M}_e\}$. It is obvious that $m_e\leq |E_N(H)|-1$ since $e\notin \cup_{M\in \mathcal{M}_e}E(M)$.

\begin{claim}\label{claim3}
$\max\{m_e : e\in E_N(H)\}=|E_N(H)|-1$.
 \end{claim}

\noindent {\it Proof of Claim \ref{claim3}.}  Let $e_0\in E_N(H)$ and $M_{01}\in \mathcal{M}_{e_0}$ such that $m_{e_0}=|E_N(H)\cap E(M_{01})|=\max\{m_e : e\in E_N(H)\}$.
Suppose that $m_{e_0}<|E_N(H)|-1$. Then there is a member $M_{02}\in \mathcal{M}_{e_0}-\{M_{01}\}$ such that $M_{02}$ contains a negative edge $e_1$ of $H$.

By Claim \ref{claim4}, $S_H(e_1)\subseteq E(M_{02})$ and there is a member $M_{11}\in \mathcal{M}_{e_1}$ such that $S_H(e_0)\subseteq E(M_{11})$. So
$$\{e_0\}\cup E(M_{01})\subseteq S_H(e_0)\cup (\cup_{M\in \mathcal{M}_{e_0}-\{M_{02}\}}E(M))\subseteq E(M_{11}),$$
which implies that
$$m_{e_1}\ge |E_N(H)\cap E(M_{11})|\ge 1+|E_N(H)\cap E(M_{01})|=1+m_{e_0}.$$
This contradicts the choice of $e_0$, and so the claim holds.~$\Box$

\medskip
By Claim \ref{claim3}, there is an edge $e\in E_N(H)$ such that $E_N(H)-\{e\}$ is contained in exactly one  member of $\mathcal{M}_{e}$.
Let  $\mathcal{M}_e=\{M_1', \cdots, M_s'\}$. Without loss of generality, assume that $E_N(H)-\{e\}\subseteq E(M_1')$ and all edges of $M_i'$ ($i=2,\cdots, s$) are positive.
Since $H$ is $2$-edge-connected, it follows from the definition of $S_H(e)$ that $H/\cup_{i=1}^sM_i'$ is a circuit, and each $M_i'$  is also $2$-edge-connected.
Since $|E_N(M_1')|=|E_N(H)|-1\ge 2$ is even, $M_1'$ has a generalized barbell $H_1'$ containing all edges of $\cup_{f\in E_N(M_1')}S_{M_1'}(f)$ by Lemma \ref{LE: g-bridgeless}, and  $H_1'$
has a signed circuit double cover $\mathcal{F}_1$ by Lemma~\ref{SCDC of g-barbell}.

Since $E_N(M_1')=E_N(H)-\{e\}$ and  $S_{M_1'}(f)\supseteq S_H(f)$ for any $f\in E_N(M_1')$,
$$\cup_{f\in E_N(H)}S_H(f)\subseteq S_H(e)\cup (\cup_{f\in E_N(M_1')}S_{M_1'}(f)).$$
By Claim \ref{CL: s-bridgeless}, $H$ has a signed circuit $C$ containing all edges of $S_H(e)$, and so $\mathcal{F}=\mathcal{F}_1\cup \{C\}$ is a desired signed circuit $\{0,1,2,3\}$-cover of $H$, a contradiction.
This complete the proof of the lemma.
\end{proof}

 \begin{lemma}\label{cut}{\rm \cite{LLZ2015}}
Let $G$ be a signed graph. Then $|E_N(G)|=\epsilon_N(G)$ if and only if for every  edge cut $T$ of $G$,
  $$|E_N(G)\cap T|\leq \frac{|T|}{2}.$$
  \end{lemma}

We now prove Lemma \ref{LE: a pair}.

\medskip

\noindent {\bf Proof of Lemma \ref{LE: a pair}.} Let $G$ be an s-bridgeless signed graph with $|E_N(G)|=\epsilon_N(G)$. Without loss of generality, we further assume that $G$ is connected. Since $G$ is s-bridgeless, $|E_N(G)|\neq 1$. If $|E_N(G)|=0$, then $G$ is a $2$-edge-connected unsigned graph. The lemma is trivial, and thus assume that $|E_N(G)|\ge 2$.

 Let $$ G_1=G-B(G)- (\cup_{e\in E_N(G)}S_G(e)).$$
By the definitions of $B(G)$ and $S_G(e)$, $G_1$ contains no negative edges of $G$ and is bridgeless.

To construct $G_2$, let $H=T+E_N(G)$, where $T$ is a spanning tree of $G-E_N(G)$ (the existence of $T$ is guaranteed by Lemma \ref{cut}).
Note that we have the following simple facts:

(1) $E_N(G)=E_N(H)$;

(2) $B_g(G)\subseteq B_g(H)$;

(3) $B_s(G)\subseteq B_s(H)$;

(4) $S_G(e)\subseteq S_H(e)$ for each $e\in E_N(G)$.

By Lemma \ref{LE: s-bridgeless}, $H$ has a signed circuit $\{0,1,2,3\}$-cover $\mathcal{F}_2$ such that each edge of $B_s(H)\cup (\cup_{e\in E_N(H)}S_H(e))$ $(\supseteq B_s(G)\cup (\cup_{e\in E_N(G)}S_G(e)))$ is covered by $\mathcal{F}_2$ at least once.
 Let $G_2=G[\cup_{C\in \mathcal{F}_2}E(C)]$. Since $G$ is s-bridgeless, $B_s(G)=B(G)$, and so $E(G)=E(G_1)\cup E(G_2)$. It is obvious that $G_2-E_N(G)$ is acyclic and $\mathcal{F}_2$ is a desired signed circuit $\{1,2,3\}$-cover of $G_2$.

In particular, if $G$ is g-bridgeless with even negativeness, then $B_g(G)=B(G)$ and by Lemma \ref{LE: g-bridgeless}, $H$ has a generalized barbell, denoted by $G_2$, containing all edges of $B_g(H)\cup (\cup_{e\in E_N(H)}S_H(e))$ $(\supseteq B_g(G)\cup (\cup_{e\in E_N(G)}S_G(e)))$. Thus $E(G)=E(G_1)\cup E(G_2)$, $G_2-E_N(G)$ is acyclic and by Lemma \ref{SCDC of g-barbell}$, G_2$ has a signed circuit double cover.
This proves Lemma \ref{LE: a pair}. $\Box$

%%%%%%%%%%%%%%%%%%%%%%%%%%%%%%%%%%%%%%%%%%%%%%%%%%%%%%%%%%%%%%%%%%%%

\section{Proof of Theorem~\ref{TH: main result g}}
\label{SS: final proof}
In this section, we complete the proof of Theorems~\ref{TH: main result g} by applying Lemma \ref{LE: a pair}.
Let $G$ be an s-bridgeless signed graph with $\epsilon_N(G)> 0$.
We only need to consider the case $|E_N(G)|=\epsilon_N(G)$ since the existence and the length of a signed circuit cover are two invariants under the switching operations.

Since $G$ is s-bridgeless and $\epsilon_N(G)> 0$, we have that $|E_N(G)|=\epsilon_N(G)\ge 2$.
If $G$ contains positive loops, then we may consider the subgraph  obtained from $G$ by deleting all positive loops. Thus we further assume that  $G$ contains no positive loops.

By Lemma \ref{LE: a pair},  $G$ has a bridgeless unsigned subgraph $G_1$ and a signed subgraph $G_2$ such that $E(G_1) \cup E(G_2) = E(G)$, $G_2-E_N(G)$ is acyclic and $G_2$ has a signed circuit $\{1,2,\cdots, k\}$-cover $\mathcal{F}_2$, where $k=2$ if $G$ is  g-bridgeless  with even negativeness and $k=3$ otherwise.

 Note that $E(G_1)\subseteq G-E_N(G)$ and thus  $E(G_1)\cap E(G_2) \subseteq E(G_2)-E_N(G)$ is acyclic. Hence we have the following two inequalities.

%\begin{eqnarray*}
%E(G_1)\cap E(G_2)&\subseteq& (G-E_N(G))\cap E(G_2)\\
%                               &=&E(G_2)-E_N(G)\cap E(G_2)=E(G_2)-E_N(G),
%\end{eqnarray*}
%and so $G_{12}$ is also acyclic since $G_2-E_N(G)$ is acyclic. Since $E(G_1)\cup E(G_2)=E(G)$,
 \begin{eqnarray}
 \label{E1+E2}
 |E(G_1)|+|E(G_2)|=|E(G_1)\cup E(G_2)|+|E(G_1)\cap E(G_2)|\leq&|E(G)|+ |V(G)|-1
% &\leq&|E(G)|+ |V(G)|-1.
 \end{eqnarray}
  \begin{eqnarray}
  \label{E2}
|E(G_2)|\leq (|V(G)|-1)+|E_N(G)|=|V(G)|-1+\epsilon_N(G).
 \end{eqnarray}

 Let $\mathcal{F}_2'$ be a subset of $\mathcal{F}_2$ such that  $\mathcal{F}_2'$ is still a signed circuit cover of $G_2$ and the number of signed circuits of $\mathcal{F}_2'$ is as small as possible. We have the following claim.
 \begin{claim}
 \label{claimf2}
  $\ell(\mathcal{F}_2')=|E(G_2)|\leq k|E(G_2)|-2(k-1).$
 \end{claim}

 \noindent {\it Proof of Claim~\ref{claimf2}.}  Let $t$ be the number of signed circuits in $\mathcal{F}_2'$. Since $|E_N(G_2)|=|E_N(G)|\ge 2$, $t\ge 1$. By the choice of $\mathcal{F}_2'$,  every signed circuit in $\mathcal{F}_2'$ has an edge which is covered by $\mathcal{F}_2'$ exactly once, and so $G_2$ has at least $t$ edges which are covered by $\mathcal{F}_2'$ exactly once. Note that $k=2$ or $3$, and each signed circuit in $\mathcal{F}_2$ is of length at least $2$ since $G$ has no positive loops.
   If $t=1$, then $G_2$ is the unique signed circuit in $\mathcal{F}_2'$, and so
  $\ell(\mathcal{F}_2')=|E(G_2)|\leq k|E(G_2)|-2(k-1).$
  If $t\ge 2$, then
  $\ell(\mathcal{F}_2')\leq k(|E(G_2)|-t)+t=k|E(G_2)|-(k-1)t\leq k|E(G_2)|-2(k-1).$  $\Box$

Since $G_1$ is bridgeless and unsigned, by Theorems \ref{CC} and \ref{CCCC}, $G_1$ has a circuit cover $\mathcal{F}_1$ with total length
\begin{eqnarray}
\label{eqf1}
\ell(\mathcal{F}_1)\leq \min\{\frac{5}{3}|E(G_1)|,|E(G_1)|+|V(G_1)|-1\}.
\end{eqnarray}

 Therefore, $\mathcal{F}=\mathcal{F}_1\cup \mathcal{F}_2'$ is a signed circuit cover of $G$ and by Claim~\ref{claimf2} and Equation~(\ref{eqf1}) together with Equations~(\ref{E1+E2}) and (\ref{E2}),
 the total length of $\mathcal{F}$ satisfies that
\begin{eqnarray*}
\ell(\mathcal{F})&=&\ell(\mathcal{F}_1)+\ell(\mathcal{F}_2')\\
                         &\leq& \min\{\frac{5}{3} |E(G_1)|,|E(G_1)|+|V(G_1)|-1\}+k|E(G_2)|-2(k-1)\\
                         &\leq & \min\{\frac{5}{3}(|E(G)|+|V(G)|-1)+(k-\frac{5}{3})(|V(G)|-1+\epsilon_N(G))-2(k-1),\\
                         && (|E(G)|+|V(G)|-1)+(|V(G)|-1)+(k-1)(|V(G)|-1+\epsilon_N(G))-2(k-1)\}\\
                                                  &=& \min\{ \frac{5}{3} |E(G)|+k|V(G)|+(k-\frac{5}{3})\epsilon_N(G)-(3k-2),\\
                                                  && |E(G)|+(k+1)|V(G)|+(k-1)\epsilon_N(G)-(3k-1)\}.
\end{eqnarray*}
This completes the proof of Theorem~\ref{TH: main result g}.

 %%%%%%%%%%%%%%%%%%%%%%%%%%%%%%%%%%%%%%%%%%%%%%%%%%%%%%%%%%%%%%%%%%%%

 %%%%%%%%%%%%%%%%%%%%%%%%%%%%%%%%%%%%%%%%%%%%%%%%%%%%%%%%%%%%%%%%%%%%

%%%%%%%%%%%%%%%%%%%%%%%%%%%%%%%%%%%%%%%%%%%%%%%%%%%%%%%%%%%%%%%%%%


\begin{thebibliography}{99}

\bibitem{Alon1985SIAM} N. Alon and M. Tarsi, Covering multigraphs by simple circuits, {\em SIAM J. Algebraic Discrete Methods}, 6 (1985): 345-350.

\bibitem{Alspach1994}  B. Alspach, L.A. Goddyn, and C.-Q. Zhang, Graphs with the circuit cover property, \TransAMS 344 (1994): 131-154.

\bibitem{Bermond1983JCTB} J.C. Bermond, B. Jackson and F. Jaeger, Shortest covering of graphs with cycles, \JCTB 35 (1983): 297-308.

\bibitem{Bondy2008} J.A. Bondy and U.S.R. Murty, Graph Theory, in: GTM, vol. 244, Springer, 2008.

\bibitem{Bouchet1983JCTB} A. Bouchet, Nowhere-zero integral flows on bidirected graph, \JCTB 34 (1983): 279-292.

\bibitem{Diestel2010}
R. Diestel,  {\em Graph Theory}, Fourth edn. Springer-Verlag (2010).


\bibitem{FanJGT1994} G.-H. Fan, Short cycle covers of cubic graphs, \JGT 18 (1994): 131-141.

\bibitem{FanJCTB1994} G.-H. Fan and A. Raspaud,  Fulkerson's conjecture and circuits covers, \JCTB 61 (1994): 133-138.

\bibitem{FanJCTB1998} G.-H. Fan, Proofs of two minimum circuit cover conjectures, \JCTB 74 (1998): 353-367.

\bibitem{Itai1978} A. Itai and M. Rodeh, Covering a graph by circuits. Pages 289-299 of:  {\em Automata, Languages and Programming}. Lecture Notes in Computer Science, vol. 62. Berlin: Springer-Verlag.

\bibitem{Jackson1990} B. Jackson, Shortest circuit covers and postman tours of graphs with a nowhere-zero $4$-flow,  \SIAMC 19 (1990): 659-665.

\bibitem{Jackson1994} B. Jackson, Shortest circuit covers of cubic graphs, \JCTB 60 (1994): 299-307.

\bibitem{Jaeger1979} F. Jaeger, Flows and generalized coloring theorems in graphs. \JCTB 26 (1979): 205-216.


\bibitem{Jamshy1987JCTB} U. Jamshy, A. Raspaud and M. Tarsi, Short circuit covers for regular matroids with nowhere-zero 5-flow, \JCTB  43 (1987): 354-357.

\bibitem{Jamshy1992JCTB} U. Jamshy and M. Tarsi, Shortest cycle covers and the cycle double cover conjecture, \JCTB 56 (1992): 197-204.

%\bibitem{Kaiser2010SIAM} T. Kaiser, D. Kr\'al, B. Lidick\'y and P. Nejedl\'y, Short cycle covers of graphs with minimum degree three, \SIAMDM 24 (2010): 330-355.

\bibitem{Kostochka1995JGT} A.V. Kostochka, The $\frac{7}{5}$-conjecture strengthens itself, \JGT 19 (1995): 65-67.


\bibitem{LLZ2015} Y. Lu, R. Luo and C.-Q. Zhang, Flows of signed graphs with low edge connectivity, submitted.

\bibitem{Macajova2015JGT} E. M\'a\v{c}ajov\'a, A. Raspaud,  E. Rollov\'a and M. \v{S}koviera, Circuit covers of signed graphs, \JGT DOI: 10.1002/jgt.

\bibitem{Macajova2011JGT} E. M\'a\v{c}ajov\'a, A. Raspaud,  M. Tarsi and X.-D. Zhu, Short cycle covers of graphs and nowhere-zero flows, \JGT 68 (2011): 340-348.


%\bibitem{Macajova2015} E. M\'a\v{c}ajov\'a and M. \v{S}koviera, Characteristic flows on signed graphs and short circuit covers, arXiv:1407.5286v1.


\bibitem{Raspaud2011JCTB} A. Raspaud and X.-D. Zhu,
Circular flow on signed graphs, \JCTB 101 (2011): 464-479.

\bibitem{Seymour1981} P.D. Seymour, Nowhere-zero 6-flows, \JCTB  30 (1981): 130-135.




%\bibitem{Wei2014} E.L. Wei, W.L. Tang and D. Ye, Nowhere-zero $15$-flow in $3$-edge-connected bidirected graphs, {\em Acta Math. Sin.-English Ser.}, 30 (4) (2014): 649-660.

\bibitem{West1996}
D.B. West, {\em Introduction to Graph Theory}, Upper Saddle River,
NJ: Prentice Hall, (1996).


%\bibitem{Xu2005} R. Xu and C.-Q. Zhang, On flows in bidirected graphs, \DM 299 (2005): 335-343.


% \bibitem{Zhang1997} C.-Q. Zhang, Integer Flows and Cycle Covers of Graphs, Marcel Dekker Inc., New York, ISBN: 0-8247-9790-6, 1997.

% \bibitem{Zhang2012} C.-Q. Zhang, Circuit double covers of graphs, Cambridge University Press, 2012.

\bibitem{Zhang1990JGT} C.-Q. Zhang, Minimum cycle coverings and integer flows, \JGT 14 (1990): 537-546.


\end{thebibliography}
\end{document}